\documentclass{article}
\usepackage{amsmath,amssymb,amsfonts,amsthm}
\usepackage{mathtools,mathrsfs}
\usepackage{color}
\usepackage{cleveref}
\usepackage[a4paper]{geometry}

\usepackage{tikz}
\usepackage{pgfplots}
\usepackage{pgfplotstable}
\pgfplotsset{compat=1.16}

\newcommand{\norm}[1]{\lVert #1 \rVert}

\newtheorem{lemma}{Lemma}[section]
\newtheorem{theorem}[lemma]{Theorem}

\theoremstyle{remark}
\newtheorem{remark}[lemma]{Remark}

\theoremstyle{definition}
\newtheorem{definition}[lemma]{Definition}

\newcommand{\im}{i}

\newcommand{\vect}{\mathrm{vec}}

\author{
    Angelo A. Casulli\thanks{
        Scuola Normale Superiore, Pisa, Italy
        (\texttt{angelo.casulli@sns.it}).
    } 
    \and
    Leonardo Robol\thanks{
        Dipartimento di Matematica, Università di Pisa
        (\texttt{leonardo.robol@unipi.it}).
    } \thanks{The authors is member of the research group INdAM/GNCS.}
}
\title{Low-rank tensor structure preservation in 
  fractional operators by means of exponential sums}

\begin{document}
    \maketitle 

    \begin{abstract}
      The use of fractional differential equations is 
      a key tool in modeling non-local phenomena. Often, an efficient 
      scheme for solving a linear system involving the discretization of 
      a fractional operator is evaluating the matrix function $x = \mathcal A^{-\alpha} c$, 
      where $\mathcal A$ is a discretization of the classical Laplacian, and 
      $\alpha$ a fractional exponent between $0$ and $1$. In this work, we derive 
      an exponential sum approximation for $f(z) =z^{-\alpha}$ that is 
      accurate over $[1, \infty)$ and allows to efficiently approximate 
      the action of bounded and unbounded operators of this kind
      on tensors stored in a variety of low-rank formats (CP, TT, Tucker). The 
      results are relevant from a theoretical perspective as well, as they 
      predict the low-rank approximability of the solutions of these linear 
      systems in low-rank tensor formats. 
    \end{abstract}

    \section{Introduction}

    We are concerned with the problem of computing the 
    solution of a linear system $\mathcal A^{\alpha} x = c$
    for $0 < \alpha < 1$,  
    where $\mathcal A$ is a 
    \emph{Kronecker sum}:
    \begin{equation} \label{eq:kronsum}
        \mathcal A := \bigoplus_{i = 1}^d A_i = 
        \underbrace{A_1 \otimes I \otimes \ldots \otimes  I}_{d\ \text{terms}} + 
        I \otimes A_2 \otimes \ldots \otimes I + 
        \ldots + 
        I \otimes \ldots \otimes I \otimes A_d,  
    \end{equation}
    for $i = 1, \ldots, d$. 
    This problem arises naturally when solving fractional PDEs 
    on tensorized domains \cite{ilic2005numerical,yang2011novel}, 
    for instance approximating the steady-state behavior of the 
    initial value problem 
    \[
      \frac{\partial u}{\partial t} = -(-\Delta)^{\alpha} u + f, \qquad 
      \Delta u := \frac{\partial^2}{\partial x_1^2} + \ldots 
      + \frac{\partial^2}{\partial x_d^2}, \qquad 
      u(t, x_1, \ldots, x_d) : [0, 1]^d \to \mathbb R. 
    \]
    A common approach to approximate the differential operator 
    $-(-\Delta)^{\alpha}$ is to discretize the 
    usual Laplace operator    
    $\Delta$
    as a Kronecker sum of 
    one-dimensional discretizations of the second derivative , and 
    then raise the discrete operator to the power $\alpha$
    (adjusting the sign to make it positive definite). This 
    yields a discretization of the fractional Laplacian \cite{ilic2005numerical}, and 
    whenever the domain has a tensor structure (as in the case 
    above where $\Omega = [0, 1]^d$)
    the discrete operator is a 
    Kronecker sum as in \Cref{eq:kronsum}. Such structure is directly 
    obtained if the problem is discretized through 
    finite differences, and can be recovered with finite 
    elements as well up to inverting mass matrices. 
    
    When solving for the steady state the linear system with a matrix given as a Kronecker 
    sum is the operation with higher computational cost. A similar 
    bottleneck is encountered for the treatment of the time-dependent 
    problem by implicit methods, which are unavoidable due to the 
    stiffness of the Laplace operator. 

    The case $\alpha = 1$, which corresponds to the classical
    Laplace operator,
    has been analyzed in detail in the 
    literature (see \cite{kressner2010krylov} and the references therein).
     When $d = 2$ the problem can be recast 
    as solving a linear matrix equation \cite{palitta2016matrix}
    (called 
    \emph{Lyapunov equation} if 
    $A_1 = A_2$, and \emph{Sylvester equation} otherwise). These 
    equations are often studied by reshaping the vectors 
    $x$ and $c$ into matrices $X$ and $C$, which yields 
    \[
        X A_1^T + A_2 X = C, \qquad 
        c = \mathrm{vec}(C), \quad 
        x = \mathrm{vec}(X). 
    \]
    Here, the $\mathrm{vec}$ operator stacks all the column of 
    a matrix on top of the other. In several instances of this 
    problem, the right hand side matrix $C$ is low-rank, or 
    at least numerically low-rank (i.e., with decaying singular values).
    This is the case when the right hand side is the discretization 
    of a (piece-wise) smooth function \cite{townsend2014computing}. 
    Under this assumption, the low-rank property is numerically inherited 
    by the solution $X$, which can be efficiently approximated 
    using low-rank solvers for matrix equations such as 
    rational Krylov methods \cite{simoncini2016computational,palitta2016matrix} 
    or ADI \cite{benner2009adi}

    When $d > 2$, similar results can be obtained, but the 
    derivation is more challenging. In this context one can 
    naturally reshape the vectors $x$ and $c$ as $d$-dimensional 
    tensor, for which several (non-equivalent) 
    definitions of rank are available
    \cite{grasedyck2013literature}. Similar 
    low-rank approximability results have been given in 
    \cite{kressner2010krylov}, relying on exponential sum 
    approximation (which we describe in detail in 
    \Cref{sec:expsums}). 

    Krylov projection methods can be extended to the case 
    $0 < \alpha < 1$ when $d = 2$, using the formulation of 
    the problem as the evaluation of a bivariate matrix 
    function
    \cite{kressner2011bivariate,kressner2019krylov,massei2021rational,massei2022mixed}. 
    Although in principle this approach may be used for higher 
    $d$ as well, it leads to multivariate matrix functions 
    and Tucker tensor approximation, which has an 
    exponential storage and complexity cost in $d$, and hence does 
    not solve the so-called ``curse of dimensionality'' \cite{oseledets2009breaking}. 

    Extending 
    results for tensor Sylvester equations to the case $\alpha < 1$ is 
    inherently difficult since the separability of the operator 
    is lost, and all strategies based on displacement ranks \cite{shi2021compressibility} are not 
    easily applicable. 

    In this work, we consider the use of exponential sums 
    to derive low-rank approximability results and low-rank 
    solvers for the case of a generic $d$ and $0 < \alpha < 1$. 
    Our results can be interpreted as an extension of 
    exponential sum approximation for $1/z$, 
    see for instance \cite{hackbusch2015hierarchical} and the reference therein. 

    The work is structured as follows. In Section~\ref{sec:expsums}
    we derive an exponential sum approximation for $z^{-\alpha}$ over 
    $[1, +\infty]$, and provide guaranteed and explicit error bounds. 
    We prove that this can be used to approximate the solution 
    of the linear systems $\mathcal A^{\alpha} x = c$ in a cheap 
    way. 
    In Section~\ref{sec:appr-tensors} we show that this representation 
    of the solution can be used to derive approximation results for 
    the solution in tensors in the same low-rank structured used for 
    the right hand side (Tucker, Tensor-Train, \ldots). 
    We conclude with some numerical experiments in Section~\ref{sec:numexps}, 
    and draw some final remarks in \Cref{sec:conclusions}. 

    \section{Exponential sums}
    \label{sec:expsums}

    We consider the approximation problem of determining 
    $\alpha_j, \beta_j$ such that 
    \begin{equation} \label{eq:expsum}        
        \xi^{-\alpha} \approx \sum_{j = 1}^k \alpha_j e^{-\beta_j \xi}. 
        \qquad 
        \xi \in [1, \infty). 
    \end{equation}
    Finding an expression in the above form (which we call 
    \emph{exponential sum}) allows to approximate the 
    function $z^{-\alpha}$ of a (possibly unbounded) 
    operator $\mathcal A$ expressed as a Kronecker sum 
    at a low computational cost. Indeed, if two matrices
    $A$ and $B$ commute, we have 
    $e^{AB} = e^{BA} = e^A e^B$. Since all 
    addends in a Kronecker sum
    commute we can write 
    \[
      e^{\mathcal -\beta \mathcal A} = \bigotimes_{i = 1}^d e^{-\beta A_i}, \qquad 
      \mathcal A = \bigoplus_{i = 1}^d A_i. 
    \]
    As 
    we will see in Section~\ref{sec:appr-tensors}, this is key in deriving 
    low-rank approximability bounds. 
    We may rewrite $\xi^{-\alpha}$
    in integral form as follows:
    \begin{equation} \label{eq:lt}
	\xi^{-\alpha} = \frac{1}{\Gamma(\alpha)}\int_{0}^\infty 
	\frac{e^{-t\xi}}{t^{1 - \alpha}} \ dt, \qquad 
	\xi \in \mathbb{R_+}. 
    \end{equation}
    Employing any quadrature rule for 
    approximating \eqref{eq:lt} yields an approximant of 
    $\xi^{-\alpha}$ by taking a weighted average of evaluations 
    of the integrand, which is exactly in the form 
    of \Cref{eq:expsum}. Let 
    $w_j$ and $t_j$, for $j = 1, \ldots, k$, be the weights
    and nodes of such quadrature, respectively. Then,
    \[
    \xi^{-\alpha} \approx \sum_{j = 1}^k 
    w_j \frac{e^{-t_j\xi}}{t_j^{1 - \alpha}} = 
    \sum_{j = 1}^k 
    \alpha_j e^{-\beta_j \xi}, \qquad 
    \begin{cases}
        \alpha_j = w_j \frac{t_j^{\alpha-1}}{\Gamma(\alpha)} \\
        \beta_j  = t_j    
    \end{cases}. 
    \]

    Our aim is deriving a quadrature that is uniformly accurate 
    over $[1, +\infty)$. We will achieve this 
    goal by a technique called \emph{sinc quadrature}, 
    also known as infinite trapezoidal rule, 
    coupled with appropriate change of variables. 

    We briefly recap the classical results on sinc quadrature 
    in \Cref{sec:sinc}; then, we build the 
    approximation over $[1, \infty)$ in \Cref{sec:sqrtexp}
    and we show how this can be used to approximate the 
    solution of the linear system 
    $\mathcal A^{\alpha}x = c$ 
    and to provide theoretical 
    prediction of approximability in low-rank tensor formats 
    for $x$, under the assumption that 
    $c$ is itself of low tensor rank (up to appropriately reshaping it). 

    \subsection{Sinc quadrature}
    \label{sec:sinc}

    We refer the reader 
    to \cite{stenger2012numerical} for a more detailed description
    of these results, and in particular 
    \cite[Appendix D]{hackbusch2015hierarchical} for a similar 
    derivation applied to $g(z) := z^{-1}$. 
    
    Let $d > 0$ and $g(z)$ be analytic 
    over the infinite strip 
    $
    \mathcal{D}_d := \{ z \ | \ 
    -d < \Im(z) < d
    \}
    $, 
    and such that the integral on the boundary of $\mathcal D_d$ is 
    finite, i.e., 
    \begin{equation} \label{eq:normDd}
        \norm{g}_{\mathcal D_d} := 
        \int_{\partial \mathcal D_d} |g(z)| \cdot |dz| < \infty.
    \end{equation}
    A sinc quadrature formula is obtained by approximating the
    integral of $g(z)$  
    over the real axis by an infinite trapezoidal rule with step $h$:
    \[
    \int_{-\infty}^{\infty} g(t)\ dt \approx 
    h \sum_{j \in \mathbb Z} g(jh). 
    \]
    For $h \to 0$, this quadrature converges exponentially. The constant
    in front of the convergence bound depends on the integral 
    in \eqref{eq:normDd}. More precisely, we have 
    
    \begin{theorem}[\protect{\cite[Theorem D.26]{borm2003introduction}}] \label{eq:expsumconv}
      \label{thm:sinc}
        Let $g(z)$ be holomorphic over $\mathcal D_d$. Then, 
        \[
        \left| 
        \int_{-\infty}^{\infty} g(t)\ dt -
        h \sum_{j \in \mathbb Z} g(jh)
        \right| \leq 
        \norm{g}_{\mathcal D_d} \cdot e^{-2\pi d / h}.
        \]
    \end{theorem}
    
    The above result is not of immediate practical use, since the 
    discretization of the integral requires to
    evaluate an infinite series. However,
    if $g(t)$ decays quickly enough for $|t| \to \infty$, we 
    can truncate 
    the sum and estimate the error by bounding the magnitude 
    of the dropped terms. 

    To obtain an efficient evaluation scheme, we will need to balance the
    error performed when truncating the series with the one 
    coming from the quadrature rule. Hence, the choice of the 
    number of terms to consider will automatically imply an 
    optimal step size $h$ in most cases. This will be discussed 
    in further detail in the next sections. 

    \subsection{Approximating $z^{-\alpha}$ over $[1, \infty)$}
    \label{sec:sqrtexp}

    The integral form of $\xi^{-\alpha}$ that we 
    considered in Equation~\ref{eq:lt} is defined by 
    an integral over $[0, \infty)$. This is not suitable 
    for employing sinc quadrature techniques, and therefore we need 
    to remap it as an integral over $\mathbb R$. To this aim, we 
    introduce the change of variable 
    $t = \log(1 + e^{\tau})^{\frac 1 \alpha}$;
    by a direct computation we obtain: 
    \begin{equation} \label{eq:cv1}
        \xi^{-\alpha} = \frac{1}{\Gamma(\alpha)}\int_{0}^\infty 
        \frac{e^{-t\xi}}{t^{1 - \alpha}} \ dt = 
        \frac{1}{\alpha \Gamma(\alpha)}\int_{-\infty}^\infty 
        \frac
        {e^{-\log(1 + e^{\tau})^{\frac 1 \alpha}\xi}}
        {1 + e^{-\tau}}
        \ d\tau.
    \end{equation}
    For the sake of notational simplicity, we now define 
    the following shorthand for the integrand:
    \[
        g(\tau) := \frac
{e^{-\log(1 + e^{\tau})^{\frac 1 \alpha}\xi}}
{1 + e^{-\tau}}
    \]
    We note that $g(\tau)$ implicitly depends on $\xi$, but we do not report 
    this dependency explicitly to keep the notation more readable. 
    Recall that $\mathcal D_d := \{ z \ | \ |\Im(z)| \leq d \} \subseteq \mathbb C$
    denotes the infinite horizontal strip of width $2d$, 
    centered around the real line.

    To use the results on sinc approximation, we first 
    need to ensure that the integrand is analytic on 
    the infinite strip $\mathcal D_d$, for suitable choices 
    of $d$. 

    \begin{lemma} \label{lem:analiticity}
      The function $g(\tau)$ is analytic on $\mathcal D_d$ 
      for any $d < \pi$. 
    \end{lemma}

    \begin{proof}
      To ensure the analiticity of 
      the integrand $g(\tau)$ we choose to exclude points where $e^{-\tau} = -1$, 
      which would force the denominator to vanish, to exclude points 
      $1 +  e^{\tau} \in \mathbb R_-$, which would force the 
      logarithm to be evaluated at its branch cut, and finally 
      to exclude all points in $\mathbb R^{-}$ from the 
      argument of the fractional power, to avoid the analogous problem for
      the logarithm implicitly defining it. If these three 
      conditions are met, then the 
      function is obtained through compositions of functions that 
      are analytic on the entire domain of interest. 

      We shall deal with these cases separately. 
      The first condition is linked with a 
      class of poles encountered for 
      $\tau = i(2k + 1) \pi$, for any $k \in \mathbb Z$, and 
      we can exclude them by requiring $d < \pi$. Similarly, 
      this condition automatically implies that 
      $1 + e^{\tau} \not \in \mathbb R_-$, which excludes evaluations 
      of $\log(1 + e^{\tau})$ on its branch cut. 

      The third situation 
       is encountered when 
      $\log(1 + e^\tau) \in (-\infty, 0]$, which in turn 
      implies $e^{\tau} \in [-1, 0)$. If we write 
      $\tau = \alpha + i\beta$, this only happens when 
      \[
        \alpha \leq 0, \qquad \beta = (2k+1) \pi, \  
        k \in \mathbb Z. 
      \]
      Similarly to the previous case, 
      we can avoid this situation by imposing 
      a constraint on $d$, and requiring
      $|\beta| \leq d < \pi$. 
    \end{proof}

    We will now derive a bound for the
    integral of the modulus 
    of $g(\tau)$ in \eqref{eq:cv1} 
    over $\partial \mathcal D_d$. This will impose 
    further constraints on the choice of $d$, which will 
    be stronger than the ones imposed 
    by \Cref{lem:analiticity}. We make the following claims, that will be detailed (with
    explicit constants) in this section. Let 
    $z = \gamma \pm \im d$ a point in $\partial \mathcal D_d$. Then, 
\begin{itemize}
	\item For $\gamma \to -\infty$, the integrand behaves as 
	$\mathcal O(e^{-|\gamma|})$. 
	\item For $\gamma \to +\infty$, the integrand behaves as 
	$\mathcal O(e^{- \xi |\gamma|^{\frac 1 \alpha} 
		\cos{\frac{d}{\alpha |\gamma|}}})$. 
\end{itemize}

We shall summarize these results in Lemma~\ref{bound-expdecay-le0}
and \ref{bound-expdecay-ge0}, that will 
be later leveraged to prove the convergence of the 
exponential sum approximation. We shall see that 
combining the hypotheses of these results, we will 
need to ensure that 
$d$ is chosen as $d \le \alpha \pi / 8$. 

\begin{lemma}\label{bound-expdecay-le0}
  Let $\tau = \gamma \pm id$, with 
	$\gamma \leq 0$, $0 < \alpha < 1$, and $\xi > 0$ be real numbers, and let $0\le d \le \frac \pi 2$ be such that 
	\begin{equation}
  \label{eqn:bound-d-gammle0}
		\sin d\le \frac{1}{4}\tan\left(\frac{\alpha\pi}{2}\right).
	\end{equation} Then, 
	\[
    |g(\tau)| \leq \left| \frac{1}{1 + e^{-\tau}}\right| \leq e^{-|\gamma|}. 
  \]
\end{lemma}
\begin{proof}
  To prove the result we show that 
  \begin{equation} \label{eq:num-le1}
    \left| 
      e^{-\xi \log(1 + e^{\gamma \pm \im d})^{\frac 1 \alpha}}
    \right| \leq 1. 
  \end{equation}
  If the above condition is satisfied, 
  using $d\le \frac{\pi}{2}$ we have 
  \begin{equation*}
    |g(\tau)| \leq 
    \frac{1}{|1 + e^{-\gamma \pm \im d}|}=
    \frac{1}{\sqrt{1+e^{-2\gamma}+2e^{-\gamma} \cos(d)}}\le 
    e^{-|\gamma|}.
  \end{equation*}
  We now prove the claim in Equation~\eqref{eq:num-le1}.
	Using polar coordinates we can write
	\[
	\log(1 + e^{\gamma \pm id}) =
	\sqrt{
		\frac 14 \log(
		1 + e^{2\gamma} + 2e^{\gamma} \cos(d)
		)^2 + \arctan \left(
		\frac
		{e^\gamma \sin d}
		{1 + e^\gamma \cos d}
		\right)^2
	} \cdot e^{i \theta(\gamma)}, 
	\]
	where
	\[
	\theta(\gamma) := \arctan\left(
	\frac
	{\pm 2\arctan \left(
		\frac
		{e^\gamma \sin d}
		{1 + e^\gamma \cos d}
		\right)}
	{\log(1 + e^{2 \gamma} + 2e^\gamma \cos(d))}
	\right).
	\]
	We can write the $\frac{1}{\alpha}$-th power of the logarithm as
	\[
	\log(1 + e^{\gamma \pm id})^{\frac{1}{\alpha}} =
	\left(
	\frac {1}{4} \log(
	1 + e^{2\gamma} + 2e^{\gamma}\cos(d) 
	)^2 + \arctan \left(
	\frac
	{e^\gamma \sin d}
	{1 + e^\gamma \cos d}
	\right)^2
	\right)^{\frac{1}{2\alpha}} \cdot e^{i\frac{ \theta(\gamma)}{\alpha}}.
	\]
	Since $\xi> 0$, it is sufficient to prove that the real 
    part of the above expression is positive. This is 
    equivalent to imposing that 
	$ 
        \cos\left(\frac{\theta(\gamma))}{\alpha}\right)\ge 0.
	$
	In particular we can show that
	\begin{equation}\label{eqn:equiv-thesis-lemma}
		\left|\frac{\theta(\gamma)}{\alpha}\right|=
		\frac{1}{\alpha}\arctan\left(
		\frac
		{2 \arctan \left(
			\frac
			{e^\gamma \sin d}
			{1 + e^\gamma \cos d}
			\right)}
		{\log(1 + e^{2 \gamma} + 2e^\gamma\cos(d))}
		\right)
	 \le \frac{\pi}{2}.
	\end{equation}
  The second inequality is equivalent to imposing 
  \[
    \frac
		{2 \arctan \left(
			\frac
			{e^\gamma \sin d}
			{1 + e^\gamma \cos d}
			\right)}
		{\log(1 + e^{2 \gamma} + 2e^\gamma\cos(d))} \leq \tan 
    \left(
      \frac{\pi \alpha}{2}
    \right). 
  \]
	Recalling that $\arctan(x)\le x$ for all $x\ge 0$, we have
	\[
	\arctan \left(
	\frac
	{e^\gamma \sin d}
	{1 + e^\gamma \cos d}
	\right)\le 	\frac{e^\gamma \sin d}{1 + e^\gamma \cos d}\le e^\gamma\sin d
	\]
	and using the 
  inequality
   $\log(1+x)\ge x-\frac{1}{2}x^2$ for $x\ge 0$, we have
	\[
	\log(1 + e^{2 \gamma} + 2e^\gamma\cos(d))\ge \log(1 + e^\gamma)\ge e^\gamma-\frac{1}{2}e^{2\gamma}.
	\]
	Hence,
	\[\frac{2 \arctan \left(
		\frac
		{e^\gamma \sin d}
		{1 + e^\gamma \cos d}
		\right)}
	{\log(1 + e^{2 \gamma} + 2e^\gamma\cos(d))}\le  \frac{2\sin d}{1-\frac{1}{2}e^\gamma}\le 4\sin d.
	\]
  We conclude by 
	using the hypothesis \eqref{eqn:bound-d-gammle0}, 
  which implies that the right hand side is bounded by 
  $\tan 
  \left(
    \frac{\pi \alpha}{2}
  \right)$, as needed. 
\end{proof}

The next result controls the size of the integrand 
when the real part of the integration variable 
$\tau = \gamma \pm \im d$ is positive, which will enable to 
bound the norm of the integral in the right half plane.  

\begin{lemma}\label{bound-expdecay-ge0}
	Let $\tau = \gamma \pm id$, with 
  $\gamma, \xi > 0$ and $0 < \alpha < 1$ real positive numbers, and 
	$0 \leq d < \frac{\alpha \pi}{4}$; then, t
  the function $g(\tau)$ is bounded above in modulus by
	\[
	|g(\tau)| \leq \Big| e^{-\xi 
		\log(1 + e^{\gamma \pm \im d})^{\frac 1 \alpha}}
	\Big| 
	\leq 
	e^{
		-\xi |\gamma|^{\frac 1 \alpha} 
		\cos\left(\frac{d}{\alpha \max\{\gamma,\frac 12\}}\right)
  }. 
	\]
\end{lemma}

  \begin{proof}
  We may write 
  \[
    |g(\tau)| = \frac{
      \Big| e^{-\xi 
        \log(1 + e^{\gamma \pm \im d})^{\frac 1 \alpha}
      }
      \Big|
    }{
      |1 + e^{-\tau}|
    } \leq \frac{
      \Big| e^{-\xi 
        \log(1 + e^{\gamma \pm \im d})^{\frac 1 \alpha}
      }
      \Big|
    }{
      |1 + e^{-\gamma} \cos(d)|
    } \leq 
    \Big| e^{-\xi 
        \log(1 + e^{\gamma \pm \im d})^{\frac 1 \alpha}
      }
      \Big|, 
  \]
  thanks to $\cos(d) \geq 0$. 

  We now prove the second inequality; 
  thanks to $\xi \in \mathbb R$, 
	\[
	\Big| e^{-\xi 
		\log(1 + e^{\gamma \pm \im d})^{\frac 1 \alpha}
	}
	\Big| = 
	\exp( -\xi \Re(
	\log(1 + e^{\gamma \pm \im d})^{\frac 1 \alpha}
	)
	). 
	\]
	Hence, in order to devise an upper bound for the left
	hand side, we need a lower bound for the
	real part of the  logarithm in the right hand side. By
	writing the argument of the logarithm in polar coordinates
	we obtain the following expression:
	\[
	\log(1 + e^{\gamma \pm id})
	= \frac 12 \log(
	1 + e^{2\gamma} + 2e^{\gamma} \cos(d)
	)
	\pm \im \arctan \left(
	\frac
	{e^\gamma \sin d}
	{1 + e^\gamma \cos d}
	\right). 
	\]
	We now rewrite the above in polar coordinates, which yields 
	\[
	\log(1 + e^{\gamma \pm id}) =
	\sqrt{
		\frac 14 \log(
		1 + e^{2\gamma} + 2e^{\gamma} \cos(d)
		)^2 + \arctan \left(
		\frac
		{e^\gamma \sin d}
		{1 + e^\gamma \cos d}
		\right)^2
	} \cdot e^{i \theta(\gamma)}, 
	\]
	where
	\[
	\theta(\gamma) := \arctan\left(
	\frac
	{\pm 2\arctan \left(
		\frac
		{e^\gamma \sin d}
		{1 + e^\gamma \cos d}
		\right)}
	{\log(1 + e^{2 \gamma} + 2e^\gamma\cos(d))}
	\right). 
	\]
	This gives an explicit expression for 
	the real part of the above logarithm raised 
	to the power $\frac 1\alpha$, which is:
	\[
	\Re(\log(1 + e^{\gamma \pm id})^{\frac 1 \alpha}) =
	\left[
	\frac 14 \log(
	1 + e^{2\gamma} + 2e^{\gamma} \cos(d)
	)^2 + \arctan \left(
	\frac
	{e^\gamma \sin d}
	{1 + e^\gamma \cos d}
	\right)^2
	\right]^{\frac{1}{2\alpha}} 
	\cos\left(\frac{\theta(\gamma)}{\alpha}\right).
	\]
	The above yields an exact expression for the quantity
	that we need to bound. We now make some simplifications, 
	employing the following inequalities:
	\begin{align} \label{eq:ineq-1}
		\log(1 + e^{2\gamma} + 2e^\gamma\cos(d)) &\geq \max \{
		2\gamma, 
		1
		\}
		&
		0 &\leq \arctan \left(
		\frac
		{e^\gamma \sin d}
		{1 + e^\gamma \cos d}
		\right) \leq d.
	\end{align}
	The two inequalities can be combined to show that 
	\[
	0 \leq  \theta(\gamma) \leq \frac{d}{\max\{ 
		\gamma, \frac 12 \}
	} \implies 
	\cos\left(
	\frac{\theta(\gamma)}{\alpha}
	\right) \geq \cos \left(
	\frac{d }{\alpha \max\{\gamma, \frac 12\}}
	\right), 
	\]
	where we used that 
	$0 \leq \theta(\gamma) \leq \frac{\alpha \pi}{2}$
	in view of $d \leq \frac{\pi \alpha}{4}$. We now 
	make use again of \eqref{eq:ineq-1} to bound the 
	first factor, obtaining 
	\[
	\Re(\log(1 + e^{\gamma \pm id})^{\frac 1 \alpha}) \geq 
	\gamma^{\frac 1 \alpha} \cos \left(
	\frac{d}{\alpha \max\{\gamma, \frac 12\}}
	\right), 
	\]
	which implies the sought bound. 
\end{proof}

Even though we have made some
simplifications in the expressions, the asymptotic behavior
for $\gamma \to \pm \infty$ 
is tight. In addition, for the values of $\gamma$ close to 
zero, the bound is still quite descriptive of the actual
behavior, as we show in Figure~\ref{fig:boundgamma} for a few 
different values of $\alpha$.  

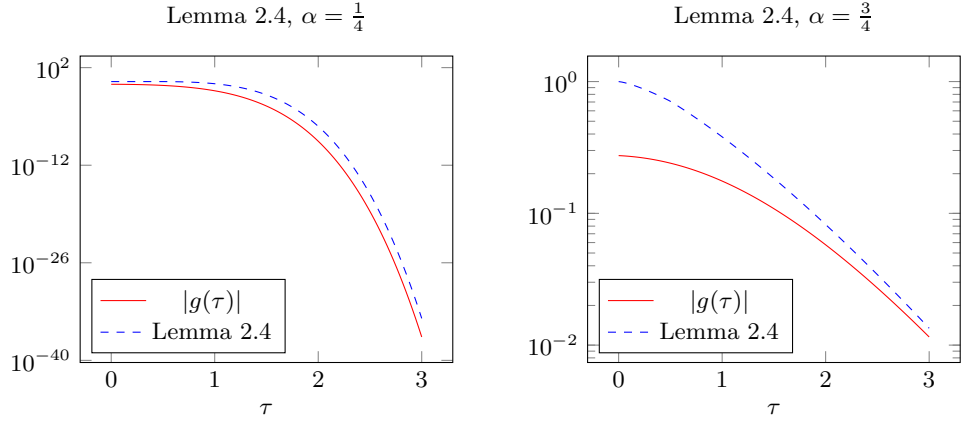
\begin{figure}
  \centering \small
  \begin{tikzpicture}
    \begin{semilogyaxis}[scale=.75, height = 7cm,
      title = {\Cref{bound-expdecay-ge0}, $\alpha = \frac{1}{4}$},
      xlabel = {$\tau$},
      legend pos = south west
    ]
      \addplot[red] table {gamma1.dat};
      \addplot[blue, dashed] table[x index = 0, y index = 2] {gamma1.dat};
      \legend{
        $|g(\tau)|$, 
        \Cref{bound-expdecay-ge0}
      }
    \end{semilogyaxis}
  \end{tikzpicture}~~~~~~~
  \begin{tikzpicture}
    \begin{semilogyaxis}[scale=.75, height = 7cm, 
      title = {\Cref{bound-expdecay-ge0}, $\alpha = \frac{3}{4}$},
      xlabel = {$\tau$},
      legend pos = south west
    ]
      \addplot[red] table {gamma3.dat};
      \addplot[blue, dashed] table[x index = 0, y index = 2] {gamma3.dat};
      \legend{
        $|g(\tau)|$, 
        \Cref{bound-expdecay-ge0}
      }
    \end{semilogyaxis}
  \end{tikzpicture}
  \caption{Bounds for the modulus of $g(\tau)$, obtained 
  for $\alpha \in \{\frac{1}{4}, \frac{3}{4}\}$ by \Cref{bound-expdecay-ge0}. 
  The value of $d$ in these examples is chosen as $d = \frac{\pi}{16}$. }
  \label{fig:boundgamma}
\end{figure}

We now have all the tools to give an explicit upper bound to 
the integral of the modulus of $g(\tau)$ over the boundary 
of $\mathcal D_d$. 

\begin{lemma}
  \label{lem:Dd-sqrtexp}
  For any $d$ satisfying $0 < d \leq \frac{\pi\alpha}{8}$
  with $0 < \alpha < 1$,  
  it holds:
  \[
    \norm{g}_{\mathcal D_d} = 
    \int_{\partial \mathcal D_d} |g(\tau)| \cdot |d\tau| \leq 
    2\left(
       1 + \log(2)+\frac{\Gamma(\alpha+1)}{(\xi\cos(\frac{\pi}{8}))^\alpha}
      \right).
  \]
\end{lemma}

\begin{proof}
  First, we note that for any $d$ in the region
  of interest we have
  \[
    \sin(d) \leq d \leq \frac{\pi\alpha}{8} \leq \frac 14 
      \tan\left(
        \frac{\pi\alpha}{2}
      \right), 
  \]
  and therefore the hypotheses of \Cref{bound-expdecay-le0}
  and \Cref{bound-expdecay-ge0} are satisfied. In addition, 
  thanks to the property $|g(\tau)| = |g(\overline \tau)|$, we 
  may rewrite the integral as 
  \[
    \int_{\partial \mathcal D_d} |g(\tau)| \cdot |d\tau| = 
    2 \int_0^\infty |g(\gamma + \im d)| d\gamma + 
    2 \int_{-\infty}^0 |g(\gamma + \im d)| d\gamma. 
  \]
  The integrands can be dealt with separately. In $(-\infty, 0]$ 
  we can use \Cref{bound-expdecay-le0} to obtain 
  the following bound:
  \[
    2 \int_{-\infty}^0 |g(\gamma + \im d)| d\gamma \leq 
    2 \int_{-\infty}^0 \frac{1}{1 + e^{-\tau}} d\gamma = 2\log(2).
  \]
  Similarly, we can bound the integral between $0$ and 
  $\infty$ as follows:
  \begin{align*}
    2 \int_0^\infty |g(\gamma + \im d)| d\gamma &\leq 
    2\int_0^{1} d\gamma + 
    2\int_{1}^{\infty} e^{-\xi |\gamma|^{\frac 1\alpha} \cos(\frac{d}{\alpha})} d\gamma  \\
    &\leq 2 + \frac{2\Gamma(\alpha+1)}{(\xi\cos(\frac d \alpha))^\alpha} 
    = \leq 2 + \frac{2\Gamma(\alpha+1)}{(\xi\cos(\frac{\pi}{8}))^\alpha}, 
  \end{align*}
  where in the last inequality we have used $d \leq \frac{\pi\alpha}{8}$. 
  The result follows by combining these two bounds. 
\end{proof}

\begin{remark}
  The bound for the integrand in $[0, \infty]$ is not asymptotically 
  sharp as used in the proof of \Cref{lem:Dd-sqrtexp}, since 
  for $\gamma \to \infty$ we have 
  $\cos(\frac{d}{\alpha \gamma}) \to 1$, and instead we have 
  replaced it with $\cos(\pi/8) \approx 0.9238795\ldots$; however, 
  this does not make a dramatic difference in practice, and makes 
  the result much more readable. 
\end{remark}

Thanks to the estimate of \Cref{lem:Dd-sqrtexp}, we may now 
approximate $\xi^{-\alpha}$ with an infinite series as follows:
\[
  \xi^{-\alpha} = h \sum_{j \in \mathbb Z} g(jh) + \epsilon_h, \qquad 
  |\epsilon_h| \leq 2\left(
    1 + \log(2)+\frac{\Gamma(\alpha+1)}{(\xi\cos(\frac{\pi}{8}))^\alpha}
   \right) e^{-2 \pi d / h}. 
\]
However, this not yet a practical algorithm, since we need to 
truncate the series to a finite sum. We use the following 
notation:
\begin{equation*}
	E(g,h)=h\sum_ {j\in \mathbb {Z}}g(jh) \quad \text{and} \quad
	 E_{N_-,N_+}(g,h)=h\sum_{j=-N_-}^{N_+} g(jh). 
\end{equation*}
We need an estimate for the error introduced by truncating 
the sum to $N_+$ positive terms, and $N_-$ negative ones. We 
state the following lemma, which is tailored to the decay 
properties of the function $g(\tau)$ considered in this section. 
\begin{lemma} \label{lem:bound-NmNp}
  Let $c_-,c_+$ and $\beta$ be positive constants such 
  that 
	\begin{equation}\label{eqn:bound g-}
		g(x)\le c_-e^{-|x|} \quad \text{for } x \le 0,
	\end{equation}
	and
	\begin{equation}\label{eqn:bound g+}
		g(x)\le c_+e^{-\beta|x|^{\frac{1}{\alpha}}} \quad \text{for }x\ge 0.
	\end{equation}
  Then, the remainder $E(g,h)-E_{N_-,N+}(g,h)$ satisfies:
	\begin{equation}\label{eqn:bound E-En}
		\left|E(g,h)-E_{N_-,N+}(g,h)\right|\le c_-\frac{ e^{-N_-h}}{h} + c_+\frac{\alpha e^{-\beta (N_+ h)^ {\frac{1}{\alpha}}}}{\beta h^{\frac{1}{\alpha}}},
	\end{equation}
\end{lemma}
\begin{proof} Since
	\begin{equation*}
		E(g,h)-E_{N_-,N+}(g,h)=\sum_{k>N_-}g(-kh)+\sum_{k>N_+}g(kh),
	\end{equation*}
	using \eqref{eqn:bound g-} and \eqref{eqn:bound g+} we have
	\begin{equation*}
		\begin{aligned}
			E(g,h)-E_{N_-,N+}(g,h)&\le c_-\sum_{k>N_-}e^{-kh}+c_+\sum_{k>N_+}e^{-\beta (kh)^{\frac{1}{\alpha}}}\\
			&\le c_-\int_{N_-}^\infty e^{-kh}dk+c_+\int_{N_+}^\infty e^{-\beta (kh)^{\frac{1}{\alpha}}}dk \\
			&= c_-\frac{e^{-N_- h}}{h}+c_+\int_{k>N_+}^\infty e^{-\beta (kh)^{\frac{1}{\alpha}}}dk.
		\end{aligned}
	\end{equation*}
	To give an upper bound to the last integral let $x=k^{\frac{1}{\alpha}},$ we have 
	\begin{equation*}
		\int_{k>N_+}^\infty e^{-\beta (kh)^{\frac{1}{\alpha}}}dk=\alpha\int_{x>N_+^{\frac{1}{\alpha}}}^\infty \frac{e^{-\beta xh^{\frac{1}{\alpha}}}}{x^{1-\alpha}}\le \alpha\int_{x>N_+^{\frac{1}{\alpha}}}^\infty e^{-\beta xh^{\frac{1}{\alpha}}}=\alpha \frac{e^{-\beta  (N_+h)^{\frac{1}{\alpha}}}}{\beta h^{\frac{1}{\alpha}}}.
	\end{equation*}
\end{proof}

We now address the problem of determining the number of 
terms required to have a prescribed accuracy $\epsilon$. 
\Cref{thm:sinc} suggests that $h$ should be chosen to 
have $e^{-2\pi d / h} \approx \epsilon$. If $N_+$ and $N_-$
are also chosen accordingly, this enables to guarantee the
required accuracy, as predicted by the following result. 

\begin{theorem} \label{thm:exp-bound}
  Let $\epsilon > 0$ and $0 < \alpha < 1$. 
  Then, for any $0 < d \leq \frac{\pi\alpha}{8}$, and 
  \[
    h = \frac{2\pi d}{\log(\epsilon^{-1})}, \qquad 
    N_- = \frac{2\pi d}{h^2}, \qquad 
    N_+ = \left(
      \frac{2\pi d h^{-\frac{\alpha+1}{\alpha}}}{\beta}
    \right)^\alpha, 
  \]
  where $\beta = \cos(2d/\alpha) \geq \cos(\pi/4)$, it holds 
  \begin{align*}
  |\xi^{-\alpha} - E_{N_-,N_+}(g, h)| &\leq 
  \left( 
         \norm{g}_{\mathcal D_d} + \frac{1}{h} + 
         \frac{1}{\beta h^{\frac{1}{\alpha}}}
      \right) \epsilon.
  \end{align*}
  If $\epsilon$ is chosen smaller than $e^{- \pi^2 / 4} \approx 0.085$, then 
  the error can be bounded by
  \begin{align*}
    |\xi^{-\alpha} - E_{N_-,N_+}(g, h)| &\leq  2 \left[
        1 + \log(2) + \frac{\Gamma(\alpha + 1)}{\cos(\pi/8)^\alpha}
        + \cos(\pi/4)^{-1} \left(
          \frac{4 \log(\epsilon^{-1})}{\pi^2 \alpha}
        \right)^{\frac 1 \alpha}
      \right] \epsilon. 
  \end{align*}
  In particular, for large $N$ and small $\epsilon$
  we the asymptotics
  $\epsilon \sim \mathcal O(e^{-\sqrt{2\pi d N}})$ 
  and $N \sim \mathcal O(\log^2(\frac{1}{\epsilon}) / 2\pi d)$ 
  hold up 
  to logarithmic factors. 
\end{theorem}

\begin{proof}
  Leveraging \Cref{thm:sinc}, we can bound the quadrature error by 
  \[
    |\xi^{-\alpha} - E(g, h)| \leq \norm{g}_{\mathcal D_d} e^{-2 \pi d / h}. 
  \]
  We now show that the proposed choices of $N_-$ and $N_+$ provide 
  an error bound with the same exponential convergence, but 
  different constants in front. Using \Cref{lem:bound-NmNp} we 
  obtain:
  \begin{align*}
    |E(g,h) - E_{N_-, N_+}(g,h)| &\leq 
      c_- \frac{e^{-N_-h}}{h} + 
      c_+ \frac{\alpha e^{-\beta (N_+h)^\frac{1}{\alpha}}}{\beta h^{\frac{1}{\alpha}}}.  \\
      & \leq \left( 
         \frac{c_-}{h} + 
         \frac{\alpha c_+}{\beta h^{\frac{1}{\alpha}}} + 
      \right) e^{-2\pi d / h}, 
  \end{align*}
  where $\beta = \cos(2d / \alpha)$ 
  applying Lemma~\ref{bound-expdecay-ge0}
  with the inequality $\max\{ \frac{1}{2}, \gamma \} \geq \frac{1}{2}$. 
  We have that $c_- = 1$ thanks to \Cref{bound-expdecay-le0} 
  and $c_+ = 1$, thanks to
  \Cref{bound-expdecay-ge0}. The final bound is obtained 
  using the explicit expression for $\norm{g}_{\mathcal D_d}$ together with 
  \[
    \epsilon \leq e^{-\pi^2 / 4} \implies 
    h = \frac{2 \pi d}{\log(\epsilon^{-1})} \leq 
    \frac{\pi^2 \alpha}{4 \log(\epsilon^{-1})} \leq 
    \frac{\pi^2}{4 \log(\epsilon^{-1})} \leq 1, 
  \]
  which implies $\frac{1}{\beta h^{frac 1 \alpha}} \geq \frac{1}{h}$, and 
  therefore allows to give the upper bound 
  \[
    \frac{1}{h} + \frac{1}{\beta h^{\frac 1 \alpha}} \leq 
    \frac{2}{\beta h^{\frac 1 \alpha}} \leq \frac{2}{\cos(\pi/4)} \left(
      4 \frac{\log(\epsilon^{-1})}{\pi^2 \alpha}
    \right)^{\frac 1 \alpha}.
  \] 
  The claim on the asymptotic growth for $\epsilon,h \to 0$ 
  follows by noting that the dominant term in $N$ is 
  $N_- \sim \mathcal O(h^{-2})$. 
\end{proof}

We now verify the convergence predicted by these results
by considering different $\xi \in [1, \infty)$, 
logarithmically spaced on $[1, 10^6]$. For these 
values, we compute the exponential sum approximating $\xi^{-\alpha}$
for $\alpha \in \{ 0.25, 0.75 \}$. The results, including 
the asymptotic bound from \Cref{thm:exp-bound}, are 
reported in \Cref{fig:exp1}. It is visible how \Cref{thm:exp-bound}
accurately describes the asymptotic rate of convergence 
of the approximation. 

We note that to reach machine precision 
a non trivial amount of exponentials has to be computed. 
In addition, when $\alpha$ is small $d$ has to be chosen 
small as well, obtaining a slower convergence speed, as 
predicted by \Cref{thm:exp-bound}. 

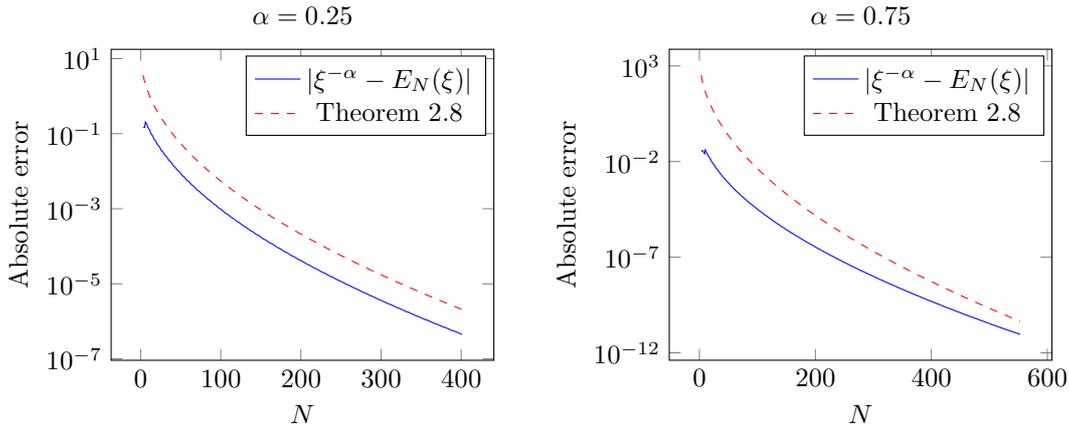
\begin{figure} \centering
  \begin{tikzpicture}
    \begin{semilogyaxis}[
      width=.45\linewidth,
      xlabel = {$N$}, ylabel = {Absolute error}, 
      title = {$\alpha = 0.25$}
    ]
      \addplot[blue] table {./example1_0.250000.dat};
      \addplot[dashed, red] table[y index = 2] {./example1_0.250000.dat};
      \legend{
        $|\xi^{-\alpha} - E_N(\xi)|$, 
        \Cref{thm:exp-bound}
      }
    \end{semilogyaxis}
  \end{tikzpicture}~~~~~
  \begin{tikzpicture}
    \begin{semilogyaxis}[
      width=.45\linewidth,
      xlabel = {$N$}, ylabel = {Absolute error}, 
      title = {$\alpha = 0.75$}
    ]
      \addplot[blue] table {./example1_0.750000.dat};
      \addplot[dashed, red] table[y index = 2] {./example1_0.750000.dat};
      \legend{
        $|\xi^{-\alpha} - E_N(\xi)|$, 
        \Cref{thm:exp-bound}
      }
    \end{semilogyaxis}
  \end{tikzpicture}

  \caption{Absolute errors of the exponential sum approximation 
  for $\xi^{\alpha}$. The error is the maximum 
  for $100$ logarithmically spaced samples over $[1, 10^6]$. 
  The value of $\alpha$ is chosen as $0.25$ and $0.75$. The convergence 
  speed for $0.25 \leq \alpha \leq 0.75$ interpolates 
  these two examples. 
  }
  \label{fig:exp1}
\end{figure}

    \section{Low-rank approximability}
    \label{sec:appr-tensors}    

    We now make use of the results developed in \Cref{sec:expsums}
    to prove that the solutions of Kronecker-structured linear systems 
    inherit 
    the low-rank tensor structure of the right hand side. 
    Recall that we are interested in linear systems of the form:
    \begin{equation} \label{eq:kronsum-linsys}
      \mathcal A^{\alpha} x = c, \qquad 
      \mathcal A = \bigoplus_{i = 1}^d A_i, \qquad 
      A_i \in \mathbb C^{n_i \times n_i}. 
    \end{equation}
    where as in \eqref{eq:kronsum}
    the ``$\oplus$'' symbol denotes 
    the \emph{Kronecker sum}.
    The vectors $x$ and $c$ may be naturally reshaped 
    into $n_1 \times \ldots \times n_d$ tensors; we denote these 
    reshaped versions with the capital letter $X, C$, respectively; we 
    will use this convention throughout the section; for instance, 
    for the vector $x$ (and tensor $X$), we have
    the correspondence:
    \[
      X \in \mathbb C^{n_1 \times \ldots \times n_d}
      \longleftrightarrow
      x = \mathrm{vec}(X) \in \mathbb C^{n_1 \ldots n_d}. 
    \]
    The linear system \eqref{eq:kronsum-linsys}
    can be rephrased as computing
    $x = f(\mathcal A) c$, where $f(z) = z^{-\alpha}$, and has 
    therefore a very natural connection with the exponential sum 
    approximation that we have discussed in the previous section.

    When dealing with high-dimensional problems (i.e., the integer
    $d$ is large) it is natural to assume that some
    low-rank structure is 
    present in the tensor $C$. If this assumption is not satisfied, 
    it is unlikely 
    that storing $C$ is possible at all. 
    
    For analogous reasons, we need to guarantee that $X$ is endowed with 
    a similar structure: otherwise, there is little hope of computing it, 
    if there is not sufficient storage for memorizing it. The exponential
    sum approximation can be used to guarantee
    that $X$ inherits the low-rank structure from 
    the right hand side $C$, and this is precisely the goal 
    of this Section. 

    In contrast to what happens with matrices, there are many 
    competing definition of a low-rank tensor. In this work, 
    we consider tensors with low CP-rank, TT-rank, and 
    multilinear rank \cite{grasedyck2013literature}. 

    We briefly recall the definition and properties of these 
    families in \Cref{sec:tensor-def}, and then show the results 
    obtainable through the exponential sum approximation in 
    \Cref{sec:appr-tensors}. 

    \subsection{Low-rank tensor formats}
    \label{sec:tensor-def}

    A natural way to define the rank of a $d$-dimensional 
    tensor $X \in \mathbb C^{n_1 \times \ldots\times n_d}$ is 
    as the \emph{minimum length of a ``low-rank decomposition''}, 
    here written for simplicity on the vectorization $x = \vect(X)$:
    \[
      x = u_{1,1} \otimes \cdots \otimes u_{1,d} + 
        \ldots + 
      u_{k,1} \otimes \cdots \otimes u_{k,d}. 
    \]
    This is usually called just \emph{tensor rank} 
    or \emph{CP rank}, and the above decomposition 
    is called a \emph{Canonical Polyadic Decomposition}
    (CPD or CP decomposition). Despite its simplicity, 
    computing such decomposition is numerically challenging for 
    large $d$ \cite{kolda2009tensor}, in contrast to what 
    happens when $d = 2$, when 
    we can leverage the singular value decomposition (SVD). 

    For this reason, several alternative definitions of 
    low-rank tensors (and the associated decompositions)
    have been introduced in recent years. We mention the 
    multilinear singular value decomposition \cite{de2000multilinear}, 
    often shortened as HOSVD (High Order SVD), and the 
    tensor train format \cite{oseledets2011tensor}. Both 
    these formats have an SVD-like procedure that allows to 
    obtain the best (or at least quasi-optimal) 
    low-rank approximation to a tensor $X$.
    To discuss the properties of these formats, 
    we shall 
    introduce the definition of 
    unfolding. 

    \begin{definition}
      The $i$-th mode unfolding $X^{(i)}$ 
      of a tensor $X$ is the matrix 
      obtained by stacking the 
      vectors obtained by collecting in a 
      matrix 
      the vectors $X_{j_i}$ with 
      all entries of the tensor with 
      the $i$-th index equal to $j_i$, 
      i.e., 
    \[
      X^{(i)} = \begin{bmatrix}
        & \vect(X_1)^T & \\ 
        & \vdots & \\
        & \vect(X_{n_d})^T & \\ 
      \end{bmatrix} \in \mathbb C^{n_i \times (n / n_i)}, 
    \]
    where $n = \prod_{i=1}^d n_i$. 
    \end{definition}

    The unfoldings can be used to define the multilinear rank 
    of a tensor $X$. 

    \begin{definition}[\cite{de2000multilinear}]
      The \emph{multilinear rank} of a tensor $X$ is the 
      tuple $r = (r_1, \ldots, r_d)$, where 
      $
        r_i = \mathrm{rank}(X^{(i)}), 
      $
      and $X^{(i)}$ is the $i$-th mode unfolding of $X$. 
    \end{definition}

    We often say that a tensor has multilinear rank 
    smaller than $r = (r_1, \ldots, r_d)$, to mean that 
    the rank is component-wise smaller. 
    We can use matrices to act on tensors, as described by 
    the following. 

    \begin{definition}
      Given a matrix $A \in \mathbb C^{m_j \times n_j}$ 
      and a tensor $X \in \mathbb C^{n_1 \times \ldots\times n_d}$, the 
      \emph{$j$-th mode product of $X$ times $A$}, denoted 
      by $X \times_j A$, is the $d$-dimensional 
      tensor 
      $Y \in \mathbb C^{n_1 \times \dots \times n_{j-1} \times m_j \times n_{j+1} \times \dots \times n_d}$. 
      defined by:
      \[
        Y_{i_1, \dots, i_d} = 
          \sum_{k=1}^{n_j} A_{i_j k} X_{i_1, \dots, i_{j-1}, k, i_{j+1}, \dots, i_d}. 
      \] 
    \end{definition}

    If $d = 2$ and therefore $X$ is a 
    matrix, we have $X \times_1 A = AX$ and 
    $X \times_2 A = XA^T$. Hence, this operation can be seen 
    as the high-order generalization of left and right 
    matrix multiplication.     
    We remark a few useful properties that relate 
    unfoldings and $j$-th mode products. 
    \begin{lemma} \label{lem:multilinear-rank}
      Let $Y = X \times_i A$. Then, 
      \begin{enumerate}
        \item[$(i)$] $Y^{(i)} = A X^{(i)}$;
        \item[$(ii)$] 
          $y = (\underbrace{I \otimes \ldots \otimes I}_{i-1\text{ terms}} 
          \otimes A \otimes 
          \underbrace{I \otimes \ldots \otimes I}_{d-i-1\text{ terms}}) x$;
        \item[$(iii)$] the multilinear rank of $Y$ is bounded by 
          $r = (r_1, \ldots, r_d)$, the multilinear rank of $X$;
        \item[$(iv)$] for any other tensor $Z$ with multilinear rank 
          $(s_1, \ldots, s_d)$, the multilinear rank of 
          $X + Z$ is bounded by 
          $(r_1 + s_1, \ldots, r_d + s_d)$. 
      \end{enumerate}
      where as usual $x = \vect(X)$, $y = \vect(Y)$, and 
      the Kronecker product in $(ii)$ has the only matrix different from the 
      identity in position $i$. 
    \end{lemma}

    A direct consequence of the second representation of the 
    $i$-th mode product is that, for any choice of matrices $A, B$ and 
    $i \neq j$, we have 
    $(X \times_j B ) \times_i A  = (X \times_i A) \times_j B$. 
    Hence, we avoid unnecessary brackets when combining 
    several $j$-mode products writing 
    $
      X \times_{j_1} A_{j_1} \ldots \times_{j_{\ell}} A_{j_{\ell}}
    $.

    The (quasi)-optimal multilinear rank $r = (r_1, \ldots, r_d)$
    approximant to a generic tensor $X$ can be effectively 
    computed by repeatedly truncating the $i$-th mode unfoldings;
    this procedure is usually known as \emph{multilinear SVD}, 
    or \emph{high-order SVD} (HOSVD) \cite{de2000multilinear}.

    If a tensor $X$ has a low multilinear rank, it can be 
    efficiently expressed through a \emph{Tucker decomposition}; 
    with our current notation this can be written as follows:
    \[
      X = B \times_1 U_1 \times_2 U_2 \ldots \times_d U_d, 
    \]
    where $B \in \mathbb C^{r_1 \times \ldots \times r_d}$, 
    and $U_j$ are $n_j \times r_j$ matrices with orthogonal 
    columns. When 
    the multilinear ranks are smaller than the dimensions 
    $n_1, \ldots, n_d$, this representation allows to compress 
    the data. 
    
    We remark that for very large $d$, this representation can 
    still be too expensive: even if the $r_i$ are small, they are 
    still multiplied together, and making the simplifying assumption 
    that $r := r_1 = \ldots = r_d$ the storage requirements for this 
    decomposition are $\mathcal O(r^d + (n_1 + \ldots + n_d)r)$ memory --- 
    which is exponential with respect to $d$. So, even if 
    when $r \ll n_i$ this format allows to save a large amount 
    of memory, working with general high-dimensional problems may 
    remain unfeasible. 
    
    To overcome this drawback, several other tensor formats 
    have been introduced: \emph{Tensor Trains} \cite{oseledets2011tensor}
    (also called \emph{Matrix Product States}, or MPS \cite{mps}), 
    \emph{Hierarchical Tucker Decompositions} \cite{hierarchicaltucker}, 
    and more general \emph{Tensor Networks} \cite{tensornetworks}. 

    In this work, we focus on Tensor Trains, and we briefly recap 
    the properties that are relevant for our results. The TT format 
    requires another definition of rank (the TT-ranks), which still 
    requires the introduction of appropriate matricization. 
    We expect similar result to hold for other tensor 
    formats whose ranks can be described by means of 
    matricizations (such as Hierarchical Tucker, or Tensor Networks \cite{hierarchicaltucker}). 

    Given a $d$-dimensional tensor $X$, we define the matrices 
    $X^{\{i\}}$ obtained by grouping the first $i$ indices 
    together as row indices, and the remaining ones as columns 
    indices. The vector $r = (r_1, \ldots, r_{d-1})$, where 
    $r_i$ is the rank of $X^{\{i\}}$, is called the 
    \emph{Tensor-Train rank of $X$} (or TT-rank). 

    A tensor with TT-rank smaller than $(r_1, \ldots, r_{d-1})$ 
    can be decomposed as follows \cite{oseledets2011tensor}: 
    \begin{equation} \label{eq:tt-expr}
      X_{i_1, \ldots, i_d} := 
        \sum_{s_1, \ldots, s_{d-1}}
          C^{(1)}_{i_1 s_1}
          C^{(2)}_{s_1 i_2 s_2}
          \ldots 
          C^{(d-1)}_{s_{d-2} i_{d-1} s_{d-1}}
          C^{(d)}_{s_{d-1} i_d}, 
    \end{equation}
    where $C^{(j)}$ are called \emph{carriages} and can be either 
    matrices ($j = 1, d$) or three-dimensional tensors ($1 < j < d$).
    It is readily apparent that this representation breaks the 
    so-called \emph{curse of dimensionality}: a tensor with low 
    (TT-)ranks can be stored with a number of parameters only 
    polynomial in $d$. 

    From \Cref{eq:tt-expr} we note that the operation 
    $X \times_j A$ can be efficiently evaluated in the TT-format, 
    as that only requires to modify $C^{(j)}$ by computing 
    $C^{(j)} \times_2 A$ (with the only exception $j = 1$, 
    where the required operation is $C^{(1)} \times_1 A$). Hence, 
    we may state a Tensor Train analogue of the last 
    item in \Cref{lem:multilinear-rank}.

    \begin{lemma} \label{lem:tt-rank}
      Let $Y = X \times_i A$. Then, 
      \begin{enumerate}
        \item the Tensor Train rank of $Y$ is bounded by 
        $r = (r_1, \ldots, r_{d-1})$, the 
        Tensor Train rank of $X$, 
        \item for any other tensor $Z$
          with Tensor-Train rank $(s_1, \ldots, s_{d-1})$, 
          the Tensor Train rank of $Y + Z$ is bounded by 
          $(r_1 + s_1, \ldots, r_{d-1} + s_{d-1})$. 
      \end{enumerate}
    \end{lemma}

    \begin{proof}
      The first claim follows by the current discussion, since the 
      dimensions of the updated carriage $C^{(j)} \times_2 A$ 
      involving the ranks are not modified. 
      We refer the reader to \cite[Section~4.1]{oseledets2011tensor}
      for a proof of the second one. 
    \end{proof}

    \subsection{Low-rank approximation in the symmetric positive definite case}

    We consider the case where the matrices $A_i$ defining 
    $\mathcal A$ are symmetric positive definite. On one hand, 
    this greatly simplifies the derivation of the results 
    thanks to the normality and the fact that the spectrum of 
    $\mathcal A$ is real. On the other hand, the non-negativity 
    of the eigenvalues is a common assumption when taking the 
    negative fractional power of an operator, and therefore 
    this assumption is not particularly restrictive in practice. 

    We will make repeated use of the following fact. 

    \begin{lemma} \label{lem:exp-break}
      Let $A_i, i = 1, \ldots, d$, be matrices 
      of size $n_i \times n_i$, and $X$ any
      $d$-dimensional tensor of size 
      $n_1 \times \ldots \times n_d$. Then, 
      \[
        \exp\left(
          \bigoplus_{i = 1}^d A_i 
        \right) \vect(X) = 
        \vect( X \times_1 \exp(A_1)\times_2 \ldots 
        \times_d \exp(A_d)). 
      \]
    \end{lemma}

    \begin{proof}
      The proof follows noting that the addends defining 
      $\bigoplus_{i = 1}^d A_i$ commute, and using the 
      property that if $AB = BA$ then $e^{AB} = e^{BA} = e^{A}e^B$.
    \end{proof}

    \begin{theorem} \label{thm:lrapprox}
        Let $\mathcal A = \bigoplus_{i = 1}^d A_i$ be invertible, 
        with $A_i$ symmetric 
        positive definite (possibly unbounded) operators. Let 
        $x = \vect(X)$, $c = \vect(C)$, 
        and $x = \mathcal A^{-\alpha} c$. Then, for any $N \in \mathbb N$ 
        there exists an approximant $X_N$ to $X$ such that
        $\norm{X - X_N}_F \sim \mathcal O(e^{-\frac{\pi}{2} \sqrt{\alpha N}})$ 
        and:
        \begin{itemize}
            \item if $C$ has CP rank bounded by $r$, then 
              $X_N$ has CP rank bounded by $Nr$. 
            \item if $C$ has multilinear rank bounded 
              by $(r_1, \ldots, r_d)$, then 
              the multilinear rank of $X_N$ 
              is bounded by $(Nr_1, \ldots,  Nr_d)$. 
            \item if $C$ has TT-ranks bounded by $(r_1, \ldots, r_{d-1})$
              then the approximation $X_N$ has TT-ranks bounded by 
              $(Nr_1, \ldots, Nr_{d-1})$. 
        \end{itemize}
  \end{theorem}

  \begin{proof}
    Let $f_N(\xi)$ be the exponential 
    sum approximation to $\xi^{-\alpha}$ 
    with $N = N_- + N_+ + 1$ terms 
    of the form 
    \[
      \xi^{-\alpha} \approx f_N(\xi) = E_{N_-, N_+}(g, h) = 
        \sum_{j = -N_-}^{N_+} \alpha_j e^{-\beta_j \xi}, \qquad 
        \xi \geq 1. 
    \]
    obtained from Theorem~\ref{thm:exp-bound}. Then, we define 
    the approximation $X_N$ as follows: 
    \[
      X_N = \lambda_{\min}^{-\alpha} 
        \sum_{i = -N_-}^{N_+} \alpha_{i,j} C 
        \times_1 e^{-\beta_{i,j} \lambda_{\min}^{-1} A_1} \ldots 
        \times_1 e^{-\beta_{i,j} \lambda_{\min}^{-1} A_d}, 
    \]
    where $\lambda_{\min}$ is the smallest eigenvalue of $\mathcal A$ and $N = N_- + N_+ + 1$, the amount of terms in the sum. 
    Using the definition of CP rank, 
    and \Cref{lem:multilinear-rank}, \Cref{lem:tt-rank}, 
    we make the following observations:
    \begin{itemize}
      \item If $C$ has CP rank bounded by $r$, then $X_N$ has 
        CP rank bounded by $Nr$. 
      \item If the multilinear rank of $C$ is component-wise 
        bounded by $r = (r_1, \ldots, r_d)$, then the multilinear 
        rank of $X_N$ can be controlled with 
        $(Nr_1, \ldots, Nr_d)$  ---  thanks to \Cref{lem:multilinear-rank}. 
      \item If the TT-rank of $C$ is bounded by $r = (r_1, \ldots, r_{d-1})$
        then the TT-rank of $X_N$ is bounded by 
        $(Nr_1, \ldots, Nr_{d-1})$ ---  thanks to \Cref{lem:tt-rank}. 
    \end{itemize}
    We now show that the approximation $X_N$ satisfies the sought bound. 
    Using the representation 
    $\vect(X) = x = \lambda_{\min}^{-\alpha} f_N(\lambda_{\min}^{-1}\mathcal A) c$, 
    we obtain 
    \begin{align*}
      \norm{X - X_N}_F &= 
      \norm{\mathcal A^{-\alpha} c - \lambda_{\min}^{-\alpha} f_N(\lambda_{\min}^{-1} \mathcal A) c}_2 \\
      &\leq \lambda_{\min}^{-\alpha} 
        \norm{(\lambda_{\min}^{-1} \mathcal A)^{-\alpha} - f_N(\lambda_{\min}^{-1} \mathcal A)}_2 \cdot \norm{c}_2 \\
      &\leq \lambda_{\min}^{-\alpha} 
        C e^{-\sqrt{2\pi d N}} \norm{c}_2, 
    \end{align*}
    where $C$ is the constant from the 
    $\mathcal O(e^{-\sqrt{2\pi d N}})$ term in 
    \Cref{thm:exp-bound}, and 
    we have used that $\lambda_{\min}^{-1} \mathcal A$ 
    is normal and has spectrum contained in $[1, +\infty)$. We can choose 
    $d = \frac{\pi \alpha}{8}$ in \Cref{thm:exp-bound}, and obtain the 
    sought result. By reversing the dependency between 
    $N$ and $\epsilon$, we eventually obtain 
the asymptotic $N \sim \mathcal O\left(
  {\log^2(\frac{1}{\epsilon})}/{2\pi d}
\right)$. 
  \end{proof}

  \subsection{Connection with rational approximations}

  In the matrix case ($d = 2$) bounds on the rank of the solution can be obtained by linking 
  the problem with rational approximation on the complex plane. In the special case 
  $\alpha = 1$, this links to the well-known properties of low-rank Sylvester solver 
  such as ADI, that  allows to build (explicit) approximants to the 
  solution $X$ of $AX + XB = C$ in the form 
  \[
    X - X_N = r(A) X r(-B)^{-1}, \qquad 
    r(z) = \frac{p(z)}{q(z)}, 
    \qquad 
    \mathrm{rank}(X_N) \leq N \cdot \mathrm{rank}(C), 
  \]
  where $p(z)$ and $q(z)$ are polynomials of degree at most $N+1$. Using the above 
  expression, and considering polynomials which are small on the spectrum of $A$ 
  and large on the one of $B$, allows to build low-rank approximant to $X$. 
  The problem of finding such rational functions is called a \emph{Zolotarev problem}, 
  and the solution is known explicitly 
  in terms of elliptic functions \cite{zolotarev1877application}. 

  When $\alpha < 1$ the situation is less straightforward, because an equation 
  with separable coefficients cannot be written. However, similar results can be 
  derived by using a Cauchy-Stieltjes formulation for $z^{-\alpha}$: 
  \[
    z^{-\alpha} = \frac{\sin(\alpha \pi)}{\pi} \int_{0}^{\infty} \frac{t^{-\alpha}}{t + z}\ dz. 
  \]
  This representation yields a formula for the solution $x = \vec(X)$ 
  to $\mathcal A^{\alpha} x = c$ in terms of integrals of a parameter dependent 
  family of (shifted) Sylvester equations, and this can be used to determine 
  a low-dimensional subspace where a good approximation for the solution can be 
  found. This has been exploited in \cite{massei2021rational,benzi2017approximation} 
  for constructing 
  rational Krylov methods for the case $d = 2$ and $\alpha < 1$, which predict 
  an exponential decay in the singular values (in contrast with the square root 
  exponential bound from \Cref{thm:exp-bound}). 

  These bounds depend on the condition number of $\mathcal A$ (although only logarithmically), 
  and therefore cannot be applied to unbounded operators. Since multilinear and 
  tensor-train ranks are defined by matricization, we think that a similar 
  idea may be exploited to derive bounds for these special cases for $d > 2$, 
  although to the best of our knowledge this has not been worked out explicitly
  at the time of writing. 

  A good indication in this direction is given by the numerical experiments, 
  which show a better approximability with respect to these formats than the 
  one predicted by \Cref{thm:lrapprox}. There is instead little hope to apply 
  such techniques to the CP case. 

  It is worth mentioning that the connection with rational approximant 
  of $z^{-\alpha}$ have been exploited in many works 
  \cite{aceto2017rational,aceto2019rational,aceto2022exponentially,harizanov2020survey,harizanov2020analysis,bonito2015numerical}
   for designing efficient 
  solvers for fractional differential equation. Since it relies on the 
  solution of shifted linear systems, it gives effective methods for all cases
  where the matrix is sparse. Our approach using matrix exponentials is instead 
  more practical when aiming at exploiting the Kronecker structure in the 
  operator. 

  \section{Numerical experiments} \label{sec:numexps}
  In this last section we report a few numerical experiments that further validate 
  our bounds, showing in which cases they are most descriptive. In addition, we show 
  that the exponential sum expansions yields an effective solver for problems with 
  a low-rank right hand side. 

  All numerical experiments have been run on an AMD
  Ryzen 7 3700x CPU with 
  32GB of RAM, running MATLAB 2022a with the bundled
  Intel MKL BLAS. The code for the experiments can 
  be found at 
  \texttt{https://github.com/numpi/fractional-expsums}. 

  \subsection{3D Fractional Poisson equation}
  \label{sec:fractionalpoisson}

  As a first example, we consider the solution of the
  fractional Poisson equation 
  on the 3-dimensional cube $[0, 1]^3$: 
  \begin{equation} \label{eq:poisson3d}
    \begin{cases}
      (-\Delta)^{\alpha} u = f & \text{in } \Omega \\
      u \equiv 0 & \text{on } \partial \Omega 
    \end{cases}, \qquad 
    \Omega = [0, 1]^3. 
  \end{equation}
  We discretize the domain with a uniformly spaced grid with 
  $128$ points in each direction, and the operator 
  $\Delta$ by finite differences, which yields the linear system 
  \[
    \mathcal A = \bigoplus_{i = 1}^3 A_i, \qquad 
    A_i = \frac{1}{h^2} \begin{bmatrix}
      2 & -1 \\
      -1 & 2 & \ddots \\
      & \ddots & \ddots & 1 \\
      & & -1 & 2 \\
    \end{bmatrix}. 
  \]
  where $h = \frac{1}{n-1}$ is the distance between the grid points. 
  We approximate the solution of \eqref{eq:poisson3d} by computing 
  $
    \mathbf u = \mathcal A^{-\alpha} \mathbf f$
  where $\mathbf f$ is the vector containing the evaluations 
  of $f(x,y,z) = 1 / (1 + x + y + z)$ at the internal points of the discretization grid. 
  For this example, we choose $\alpha = 0.4$. 

  In Figure~\ref{fig:poisson3d} we report the quality of the approximation
  obtained for $\mathbf u$ by using the exponential sum with 
  $N$ terms described in Theorem~\ref{thm:exp-bound}. The exact 
  solution is computed by diagonalizing $A_i$, which is 
  feasible and accurate because they are symmetric and of
  moderate sizes. 

  \begin{figure}
    \centering
    \begin{tikzpicture}
      \begin{semilogyaxis}[xlabel = $N$, ylabel = {Relative error},
        width = .65\linewidth, height = .35\textheight
      ]
      \addplot table[y index = 2] {exp_tensor_d=3_n=128.dat};
        \addplot table {exp_tensor_d=3_n=128.dat};
        \legend{$\mathcal O(e^{-\sqrt{2\pi d N}})$, Relative error}
      \end{semilogyaxis}
    \end{tikzpicture}
    \caption{Relative error on the approximation of the solution for 
      the discretization of the problem in 
      \eqref{eq:poisson3d} using $128^3$ points, with 
      $\alpha = 0.4$. 
      The error is computed using the Frobenius norm, and 
      the approximation is 
      computed using exponential sums with $N$ terms, as 
      in Theorem~\ref{thm:exp-bound}. }
      \label{fig:poisson3d} 
  \end{figure}
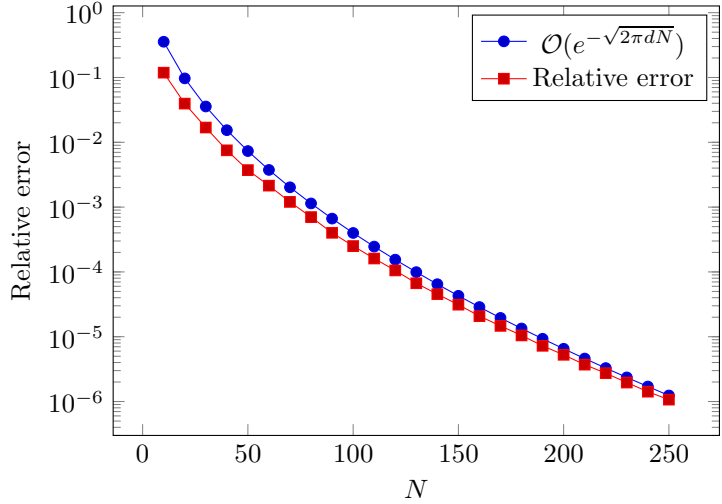

  We now consider the same example with right hand side 
  $f(x,y,z) = \sin(x)\cos(y)e^z$. Since this function is separable, 
  the corresponding vector $\mathbf f$ is the vectorization of a 
  rank $1$ tensor. This allows to directly build a low-rank 
  approximation of the solution by the expansion: 
  \begin{equation} \label{eq:lowrank-cpd}
    \mathcal A^{-\alpha} \mathbf x \approx 
      \sum_{j = 1}^N \alpha_j 
        e^{-\beta_j \mathcal A} \mathbf f, \qquad 
        \mathbf f = \vect(F), \qquad 
        e^{-\beta_j \mathcal A} \mathbf f = 
        \vect(
          F \times_1 e^{-\beta_1 A_1} \ldots \times_d e^{-\beta_d A_d}
          )
  \end{equation}
  Under these hypotheses, the cost of evaluating the inverse 
  fractional power is dominated by computing the 
  matrix exponentials, and requires $\mathcal O(dn^3 + Ndn^2)$ flops 
  for an $d$-dimensional tensor with all modes of length $n$. In 
  contrast, evaluating the fractional power by diagonalization 
  requires $\mathcal O(dn^4)$ flops. In Table~\ref{tab:lr} we 
  compare the cost of these two algorithms, using a different 
  length of the exponential sum approximation to $z^{-\alpha}$. 

  \begin{table}
    \centering
    \begin{filecontents*}{exp_tensor_lr_d=3.dat}
128	0.15059	0.061899	0.012275	0.15089	0.00012646	0.30002	1.8519e-06	0.60427	1.6235e-08
256	1.014	0.1551	0.012352	0.56977	0.00012732	0.95086	1.8645e-06	1.8016	1.6345e-08
512	8.1339	0.58708	0.012383	2.0277	0.00012776	3.6056	1.8708e-06	6.292	1.6397e-08
1024	Inf	3.1961	Inf	10.138	Inf	20.592	Inf	35.15	Inf
2048	Inf	14.912	Inf	52.305	Inf	104.26	Inf	182.12	Inf
4096	Inf	84.022	Inf	290.82	Inf	568.59	Inf	926.73	Inf
\end{filecontents*}
\pgfplotstabletypeset[clear infinite,
      every head row/.style={after row=\hline},
      empty cells with={\ensuremath{-}},
      columns={0,1,4,5,6,7,8,9},
      columns/0/.style={column name={$n$}, column type/.add={}{|}},
      columns/1/.style={column name={$t_{\text{dense}}$}},
      columns/2/.style={column name={$t_{30}$}},
      columns/3/.style={column name={$res_{30}$}},
      columns/4/.style={column name={$t_{100}$}},
      columns/5/.style={column name={$res_{100}$}},
      columns/6/.style={column name={$t_{200}$}},
      columns/7/.style={column name={$res_{200}$}},
      columns/8/.style={column name={$t_{350}$}},
      columns/9/.style={column name={$res_{350}$}},
    ]{exp_tensor_lr_d=3.dat}
    \caption{Time and accuracy of the low-rank 
      approximation to $\mathcal A^{-\alpha} \mathbf f$ obtained by 
      the exponential sums of length $N = 100, 200, 350$, and 
      runtime of the dense evaluation based on diagonalization, 
      for $d = 3$.}
    \label{tab:lr}
  \end{table}

  We note that in this case it is not practical to compute 
  the dense solution for large dimensions, since the memory 
  required is $\mathcal O(n^d)$; the low-rank approximation obtained 
  through \eqref{eq:lowrank-cpd} only requires 
  $\mathcal O(nd)$ storage. For this reason, we only report 
  the results for the dense case and the 
  accuracy up to dimension $n = 512$ in Table~\ref{tab:lr}. 

  We remark that since the convergence bound is uniform over 
  $[1, +\infty]$ the accuracy does not degrade as $n \to \infty$, 
  even if the largest eigenvalues of the discretized Laplacian 
  converge to infinity; this is necessary the case, 
  since the underlying continuous operator 
  is unbounded. 

  If $n$ grows and the $A_i$ are structured, it can be convenient 
  to exploit strategies to directly compute 
  $e^{-\beta_j A_i}v$ instead of building the entire 
  matrix exponential $e^{-\beta_j A_i}$, such as methods based 
  on Krylov subspaces (see \cite{higham2008functions} and the references therein) or 
  on truncated Taylor expansions \cite{al2011computing}. 

  \subsection{Low-rank approximability in tensor formats}

  To test the results concerning low-rank approximability, we 
  solve an equation in the form $\mathcal A^{\alpha} x = c$, and 
  then check the distance of the solution with the closest 
  rank $j$ tensor, and we compare it with the upper bound 
  from Theorem~\ref{thm:lrapprox}. We choose 
  as $A_i$ the discretization of the 1D Laplacian as in Section~\ref{sec:fractionalpoisson}, 
  and the right hand side $c$ as $c = c_1 \otimes \ldots \otimes c_d$, with 
  $c_i$ containing entries distributed as independent Gaussian random variables 
  with mean $0$ and variance $1$. 

  We have computed 
  the reference solution explicitly by diagonalization of the $A_i$. 
  Then, we have approximated 
  for each $N = 1, \ldots, 40$ the best CP approximant of rank 
  at most $N$ using the \texttt{cp\_als} algorithm in the 
  Tensor Toolbox \cite{kolda2009tensor} and 
  \texttt{cpd} from TensorLab \cite{tensorlab}, and for each 
  $N$ we have chosen the best approximation. The decay rate is compared 
  with $\mathcal O(e^{-\sqrt{2\pi dN}})$ predicted by \Cref{thm:lrapprox}
  in Figure~\ref{fig:lrapprox}. 
  The problem is chosen of size $n_1 = n_2 = n_3 = 128$, 
  the power $\alpha = \frac{1}{2}$, and 
  the tolerance for the \texttt{cp\_als} algorithm is set to $10^{-12}$, 
  and a maximum of $100$ iterations. The parameters for \texttt{cpd} have not 
  been tuned, as they were already providing good results out of the box. 

  The estimate turns out to be 
  somewhat pessimistic (the convergence of low-rank approximant in 
  CPD format is faster than what we predict), but is closer 
  than what we will obtain in the HOSVD and TT case.

  \begin{figure}
    \centering
    \begin{tikzpicture}
      \begin{semilogyaxis}[
        legend pos = south west, 
        xmax = 25, 
        title = {Distance from the best CP rank $N$ approximant},
        width=.7\linewidth, height=.35\textheight, xlabel = $N$]
        \addplot table {lowrank.dat};
        \addplot table[x index = 0, y index = 2] {lowrank.dat};
        \legend{
          $\norm{X - X_N^{(CP)}}_F$, 
          $\mathcal O(-\sqrt{2d\pi N})$
        }
      \end{semilogyaxis}
    \end{tikzpicture}

    \caption{Distance of the solution $X$ from the best approximant  
      of CP rank at most $N$, approximated by the best approximation 
      obtained from \texttt{cp\_als} in 
      the Tensor Toolbox and \texttt{cpd} from TensorLab. The distance is compared with the upper 
      bound for the asymptotic decay rate predicted by \Cref{thm:lrapprox}.}
    \label{fig:lrapprox}
  \end{figure}
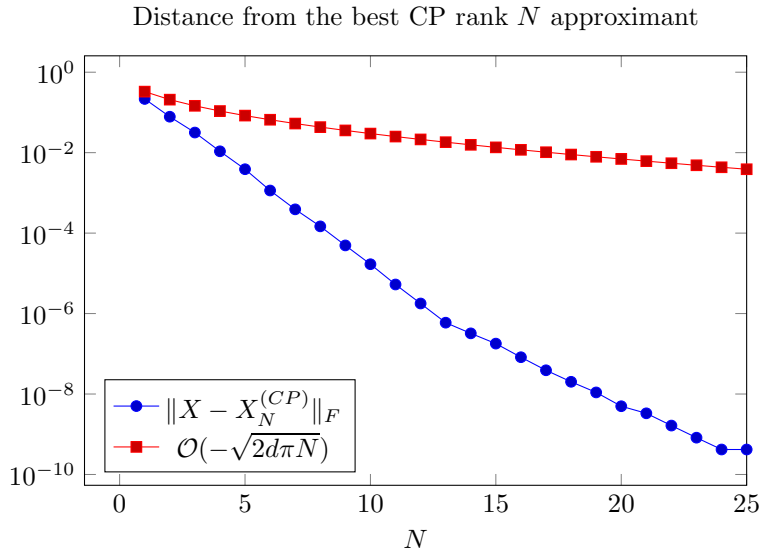

  We have then run the same tests for multilinear and Tensor-Train 
  ranks, which are much smaller. In this context, our prediction of 
  approximability turns out to be very pessimistic, as visible 
  in \Cref{fig:lrapprox2}.

  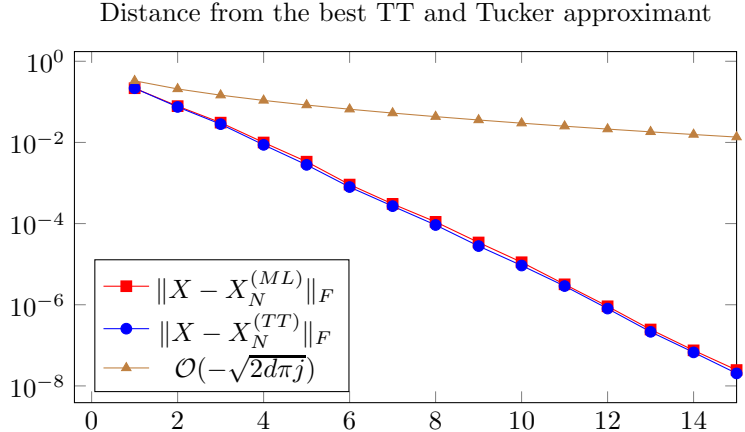
\begin{figure}
    \centering
    \begin{tikzpicture}
      \begin{semilogyaxis}[
        title = {Distance from the best TT and Tucker approximant},
        legend pos = south west, 
        width=.7\linewidth, height=.3\textheight, xmax = 15]
        \addplot[color=red, mark=square*] table[x index = 0, y index = 3] {lowrank.dat};
        \addplot [color=blue, mark=*] table[x index = 0, y index = 4] {lowrank.dat};
        \addplot [color=brown, mark=triangle*] table[x index = 0, y index = 2] {lowrank.dat};
        \legend{
          $\norm{X - X_N^{(ML)}}_F$, 
          $\norm{X - X_N^{(TT)}}_F$, 
          $\mathcal O(-\sqrt{2d\pi j})$
        }
      \end{semilogyaxis}
    \end{tikzpicture}

    \caption{Distance of the solution $X$ from the best approximant  
      of multilinear and TT ranks at most $N$, denoted by 
      $X_N^{(ML)}$ and $X_N^{(TT)}$ and approximated by \texttt{hosvd} in 
      the Tensor Toolbox and by the TT-Toolbox, respectively. 
      The distance is compared with the upper 
      bound for the asymptotic decay rate predicted by \Cref{thm:lrapprox}.}
    \label{fig:lrapprox2}
  \end{figure}

  We believe that the definition of ranks for the multilinear and 
  TT case, involving matricizations, may be analyzed with more powerful 
  tools from matrix theory, and hence obtain stronger decay bounds. 

  If on one hand the bounds are not completely descriptive of the 
  decay rate, they can be used to justify the application of low-rank 
  methods to the problems under consideration, since they provide 
  easily computable a-priori bounds. 

  \subsection{High-dimensional fractional PDEs with tensor-trains}

  We consider the computation of the solution for the solution of the 
  PDE $(-\Delta)^{\alpha} = f$ over $[0, 1]^d$, with large $d$, and we choose 
  the function $f(x_1, \ldots, x_d)$ as follows:
  \[
    f(x_1, \ldots, x_d) := \frac{1}{1 + x_1 + \ldots + x_d}, \qquad 
    x_i \in [0, 1]. 
  \]
  This function has low multilinear and tensor train ranks \cite{shi2021compressibility}, 
  but methods based on the Tucker decomposition are not suitable, because of the 
  exponential storage cost in $d$. On the other hand, the CPD of a function 
  not directly given in a separable form is not easy to compute in general. Hence, 
  we focus on solving the equation in a Tensor-Train format. 
  
  As we did in Section \ref{sec:fractionalpoisson}, we discretize the domain
  with a uniformly spaced grid with 128 points in each direction and we compute $
    \mathbf u = \mathcal A^{-\alpha} \mathbf f,$
  where $\mathcal A$ is the discretization of $-\Delta$ and $\mathbf f$ is the 
  vector containing all the evaluations 
  of $f$ at the internal points of the discretization grid.
  
  To obtain a Tensor-Train representation of $\mathbf f$, the tensor with the 
  evaluations of $f(x_1, \ldots, x_d)$ at the grid points, we relied on 
  an AMEn-based version of the TT-cross approximation, as described in 
  \cite{dolgov2014alternating}, and implemented in the TT-Toolbox. 
  This only requires to evaluate $f(x_1, \ldots, x_d)$ 
  at a few specific points in the grid, making the method 
  very effective in practice. 

  We then use our exponential sum approximation with $N$ term, 
  which requires to compute the $Nd$ matrix exponentials $e^{-\beta_j A_i}$ for 
  $j = 1, \ldots, N$ and $i = 1, \ldots, d$, 
  and then to multiply them by a low TT rank matrix with a mode-$j$ product. The latter 
  can be evaluated efficiently in the Tensor-Train arithmetic, and the storage for the result of 
  the partial sum is kept under control by recompressing the tensors with 
  economy TT-SVDs, as implemented in the \texttt{round} command of the TT-Toolbox 
  \cite{oseledets2011tensor}. 
 
 In Table \ref{tab:exp-tt} we report time and accuracy for the approximation of $\mathbf u$ by the exponential sum with $N=200$ terms. Moreover we report the TT-rank of the approximated solution.
 
   \begin{table}
    \centering
    \begin{filecontents*}{exp_tt.dat}
2	0.41601	1.645e-06	15
3	0.88138	1.7591e-06	16
4	1.8743	1.8657e-06	24
6	4.2246	nan	26
10	9.824	NaN	28
15	16.199	NaN	27
20	24.508	NaN	27
\end{filecontents*}
\pgfplotstabletypeset[clear infinite,
      every head row/.style={after row=\hline},
      empty cells with={\ensuremath{-}},
      columns/0/.style={column name={$d$}, column type/.add={}{|}},
      columns/1/.style={column name={Time(s)}},
      columns/2/.style={column name={Error}},
      columns/3/.style={column name={Rank}},
    ]{exp_tt.dat}
    \caption{Time, accuracy and rank of the final solution, of the low-rank 
      approximation to $\mathcal A^{-\alpha} \mathbf f$ on $[0,1]^d$ obtained 		by the exponential sums of length $N = 200$, and 
      different choices of $d$. The accuracy is computed by comparing the approximated solution with the one obtained by solving the Sylvester equation
      by diagonalization, which is only feasible for small $d \leq 4$. }
    \label{tab:exp-tt}
  \end{table}

  \section{Conclusions and outlook} \label{sec:conclusions}

  We have developed an exponential sum approximation for 
  $z^{-\alpha}$, which finds application in solving linear systems 
  involving fractional powers of Kronecker sums. 

  This allows to effectively solve such linear systems when the 
  right hand side is stored in any low-rank tensor formats, and 
  examples have been reported for CP, Tensor-Trains, and Tucker 
  tensors. 

  The construction also allows to predict the approximability of 
  the solution in the same format of the right hand sides. We have verified 
  that our prediction is not completely descriptive of the approximation 
  speed in the CP format and 
  for TT and multilinear ranks. We believe 
  that other tools may be used to derive better bounds in the 
  TT and Tucker case, which will be investigated in future work. In any case, 
  it can be used to provide a-prori justification for the approximability, 
  and suggests that using adaptive rank truncations may lead to very 
  good results. 

  Since the latter formats (TT and Tucker) allow for easy recompressions, the 
  proposed exponential sum approach can still be a competitive 
  solver, even if the ranks in the solution are 
  slightly overestimated. This has been demonstrated on a few practical cases. 
  In particular, in the TT case this framework allows to treat 
  very high-dimensional problems. 

  The derived bounds do not depend on the spectrum of the system 
  matrix, and can be directly applied to unbounded operators. This generality is likely 
  part of the reason why the predicted approximability in low-rank formats are worse 
  than what we obtain in practice. 

    \bibliographystyle{plain}
    \bibliography{biblio}

\end{document}